\documentclass[10pt,fleqn]{article}

\pdfoutput=1

\usepackage{graphics,graphicx}

\usepackage{mathptmx}
\usepackage[T1]{fontenc}

\usepackage{latexsym,epsfig,verbatim}
\usepackage{amsmath,amsthm,amssymb}
\usepackage{bm} 

\usepackage[symbol*,stable]{footmisc}
\DefineFNsymbols{autumn}[math]{
\circ 				{\circ\circ} 
\bullet 			{\bullet\bullet} 			
\dagger			{\dagger\dagger}
\ddagger			{\ddagger\ddagger}					
}%
\setfnsymbol{autumn}
\renewcommand\footnotemark{}

\usepackage{accents}

\usepackage{tikz}
\usetikzlibrary{matrix,arrows,decorations.pathmorphing}

\numberwithin{equation}{section} 

\theoremstyle{plain}

\newtheorem{theorem}{Theorem}
\newtheorem{lemma}[theorem]{Lemma}

\newtheorem{proposition}[theorem]{Proposition}

\theoremstyle{definition}

\newcommand{\curv}{\mathcal{C}}

\newcommand{\tangent}{\! \! \! \! \phantom{n}^\top}

\newcommand{\core}{\mathrm{core}}
\newcommand{\bT}{{\mathbb T}}

\newcommand{\qf}{\mathrm{qf}}

\newcommand{\Conf}{\mathrm{Conf}}

\newcommand{\dee}{\mathrm{d}}
\newcommand{\diam}{\mathrm{diam}}

\newcommand{\ep}{\epsilon}

\newcommand{\co}{\colon\thinspace}

\newcommand{\RR}{\mathbb{R}}

\newcommand{\MM}{\mathbb{M}}

\newcommand{\CC}{\mathbb{C}}

\newcommand{\calA}{\mathcal{A}}

\newcommand{\calE}{\mathcal{E}}

\newcommand{\calG}{\mathcal{G}}
\newcommand{\calH}{\mathcal{H}}

\newcommand{\calN}{\mathcal{N}}

\newcommand{\calP}{\mathcal{P}}

\newcommand{\calX}{\mathcal{X}}
\newcommand{\calY}{\mathcal{Y}}
\newcommand{\calZ}{\mathcal{Z}}

\newcommand{\Teich}{\mathcal{T}}

\newcommand{\grad}{\nabla}

\newcommand{\infinity}{{\rotatebox{90}{\hspace{-.75pt}\large $8$}}}

\newcommand\ME{M\kern-4pt E}
\newcommand\bME{\overline{M\kern-4pt E}}

\newcommand\nn{{\mathbf n}}

\newcommand\vv{{\mathbf v}}

\newcommand\boundary{\partial}

\def\<#1,#2>{\langle#1,#2\rangle}
\def\<#1>{\langle#1\rangle}
\newcommand{\ssm}{\smallsetminus}
\newcommand{\bdot}[1]{\overset{\large\bm .}{#1}}

\begin{document}

\title{\textbf{Skinning bounds along thick rays}}
\author{Kenneth Bromberg,
 Autumn Kent,
 and Yair Minsky\thanks{The authors were supported by NSF grant DMS--1509171, NSF CAREER Award DMS--1350075, and NSF grants DMS--1311844 and DMS--1610827.}}
 \date{March 25, 2018}

\maketitle

\begin{abstract} 
\noindent We show that the diameter of the skinning map of an acylindrical hyperbolic $3$--manifold $M$ is bounded on $\epsilon$--thick Teichm\"uller geodesics by a constant depending only on $\epsilon$ and the topological type of $\partial M$.
\end{abstract}


\section{Introduction}
Let $M$ be a compact hyperbolic $3$--manifold with totally geodesic boundary $X_M$. 
The space of convex cocompact hyperbolic metrics on the interior $M^\circ$ of $M$ is naturally identified with the Teichm\"uller space $\Teich(\boundary M)$.   Given a convex cocompact hyperbolic metric $M^X$ on $M^\circ$ associated to the marked Riemann surface $X$, the conformal boundary of $M^X$ is $X$.  The covering of $M^\circ$ corresponding to $\boundary M$ is a quasifuchsian manifold whose conformal boundary has two components: one conformally equivalent to $X$,  the other the \textit{skinning surface} $\sigma_M(X)$.  This defines a map between Teichm\"uller spaces
\[
\sigma_M \co \Teich(\boundary M) \to \Teich(\overline{\boundary M})
\]
called the \textit{skinning map}. 
See \cite{Kent.2010} for more details.

Thurston's Bounded Image Theorem \cite{Thurston.1979.Bangor} states that
\[
\diam(\sigma_M(\Teich(\boundary M))) < \infinity,
\]
and is instrumental in his proof of hyperbolization for Haken
$3$--manifolds. 
Quantitative bounds on the diameter of this map would improve our understanding of the gluing of hyperbolic structures.
One may conjecture that there is a bound
\[
\diam(\sigma_M(\Teich(\boundary M))) < {\mathcal D}
\]
where $\mathcal D$ depends only on the topological type of $\boundary M$ and not on $M$ itself.  
Our theorem supports this conjecture.

\begin{theorem}[Bound along thick rays]\label{BoundAlongThickRays.theorem}
Let $S$ be a closed orientable surface of genus greater than 1 and let $\ep > 0$.
Then there is a $\mathcal{D}$ such that if $M$ is any compact hyperbolic $3$--manifold with totally geodesic boundary $X_M \cong S$ and $\calG \co [0,\infinity) \to \Teich(\partial M)$ is any $\epsilon$--thick Teichm\"uller geodesic ray with $\calG(0) = X_M$, then
\[
\diam(\sigma_M(\calG([0,\infinity)))) \leqslant \mathcal{D}.
\]
Specifically, there are constants $A$ and $B$ depending only on $S$ and $\epsilon$ such that 
\[
\diam(\sigma_M(g([T,\infinity)))) < A e^{-B T}$ for all $T \ge 0.
\]
\end{theorem}

\subsubsection*{Sketch of the proof}

The idea of the proof is as follows, see Figure \ref{collar.figure}.

\begin{figure}
\input{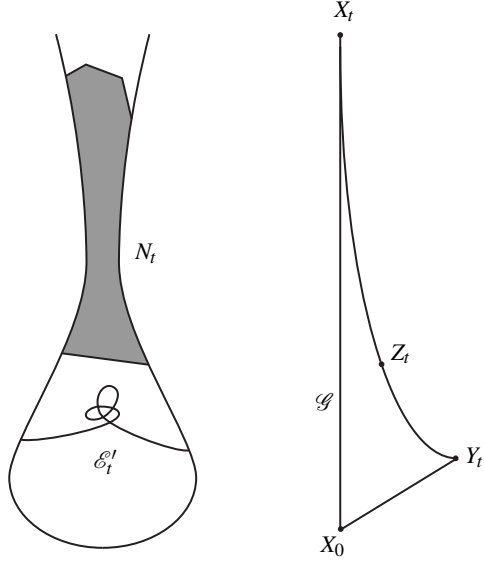}
\caption{At left is the manifold $M^{X_t}$, with surface $\calE'_t$ and collar $N_t$ in the convex core about the convex core boundary.  At right is the geodesic triangle $\triangle X_0 Y_t X_t$.}\label{collar.figure}
\end{figure}

Let $X_t$ be the surfaces along the geodesic ray, let $Y_t$ be the mirror image of the skinning surface at $X_t$, and let $M_t = M^{X_t}$ be the interior of $M$ equipped with the hyperbolic metric corresponding to $X_t$.
McMullen proved \cite{McMullen.1990} that the skinning map of an acylindrical manifold is uniformly contracting, and this means that the distance between $X_t$ and $Y_t$ is growing at a definite linear rate. 
The geodesic $[X_t, Y_t]$ from $X_t$ to $Y_t$ fellow travels our geodesic $\calG$ along a thick segment $[X_t, Z_t]$ of linearly growing length, thanks to work of Rafi \cite{Rafi.2014}.
This implies, using work of Brock--Canary--Minsky
\cite{Brock.Canary.Minsky.ELC2}, the existence of a linearly deep and
uniformly thick collar about the convex 
hull boundary of $M_t$. We establish this in Theorem
\ref{thick.collar.theorem} in Section \ref{thick.collar.section}.

In Section \ref{Geometric.inflexibility.section}, we use the Geometric
Inflexibility Theorem of Brock--Bromberg \cite{Brock.Bromberg.2011}.
This tells us that, in the complement of the thick collars of Theorem
\ref{thick.collar.theorem}, the geometry of the manifold is changing, in a $C^1$--sense, at
a rate exponentially small in $t$. 
(Here the metric distortion is measured in terms of the strain field of the family of metrics.)

We formulate two consequences of this.
Theorem \ref{eta.C1.control} gives the pointwise $C^1$ estimates in the form
that we will use. 
Theorem \ref{only.nonperipheral.thin} uses an additional estimate from
\cite{Brock.Bromberg.2011} to show 
that every peripheral curve in $M$ has an absolute lower bound on its
geodesic length along the family $M_t$.

In Lemma \ref{first.time.lemma} of Section \ref{proxy.section} we show that,
for sufficiently large $t$, there is a surface $\calE'_t$ below the
deep collar in $M_t$ that serves as a proxy for the skinning surface
$Y_t$. The surface $\calE'_t$ is the immersion in $M_t$ of a suitably smoothed
neighborhood of the convex hull boundary facing the skinning end in
the quasifuchsian cover of $M_t$.

In Lemma \ref{conformal.structure.infty} and Proposition
\ref{control.skinning} we study the relation between $\calE'_t$ and the
skinning image $Y_t$, and use it to show that the speed of $Y_t$ in
Teichm\"uller space is controlled by the $C^1$ bounds on the strain field
established in Theorem \ref{eta.C1.control}.
The surface $\calE'_t$ is uniformly thick by Theorem \ref{only.nonperipheral.thin}, and so we can apply Theorem \ref{eta.C1.control} to see that $Y_t$ moves
exponentially slowly.
It follows that the distance between $Y_0$ and $Y_t$ is uniformly bounded for all $t$.

\section{Constants, norms, and families of metrics}
\label{families.section}

Throughout the paper we will want to keep track of the dependence of constants.  
To simplify our notation we will say constants are \textit{nice} if they depend only on $\epsilon$ and the topological type of $S$.

\subsection*{Norms on tensors}
Let $V$ be a finite dimensional vector space with an inner product $g$. 
Since $V$ is finite dimensional, all norms are equivalent, and we use the {\em operator norm}. Let $x$ be a vector in $V$. Then $\|x\|^2 = g(x,x)$, and for a $(r,0)$--tensor $\tau$ on $V$, we define
\[
\|\tau\| = \underset{\|x_i\|=1}{\sup} |\tau(x_1, \dots, x_r)|.
\]
If $\tau$ is an $(r,1)$--tensor, we define
\[
\|\tau\| = \underset{\|x_i\|=1}{\sup} \|\tau(x_1, \dots, x_r)\|.
\]
If $\tau$ is an $(r,0)$-- or $(r,1)$--tensor on a Riemannian manifold $M$ then we have an operator norm $\|\tau_p\|$ at each point $p$ and we define $\|\tau\| =\sup_{p\in M}  \|\tau_p\|$.

\subsection*{Families}
If we have a $1$--parameter family of objects $ob_t$ then we write $ob = ob_0$, and $\bdot{ob}$ will denote the time--zero derivative.

\subsubsection*{Families of Riemannian metrics}
Given a smooth $1$--parameter family of Riemannian metrics $g_t$ on a manifold $M$, there is, at each time $t$, a vector--valued $1$--form $\eta_t$ defined by 
\[
\frac{\partial}{\partial t} g_t(x,y) = 2 g(\eta_t(x),y))
\]
called the \textit{strain field} at time $t$ associated to the family $g_t$.

\subsubsection*{Families of conformal structures}

A family of metrics on a surface determines a family of marked
conformal structures. Such a family is a path in Teichm\"uller
space, and we are interested in bounding the Teichm\"uller norm of its
derivative in terms of the derivative of the metrics. Namely we will
show
\begin{lemma}\label{bound.Teich.norm}
Let $g_t$ be a smooth family of Riemannian metrics on a surface
$\Sigma$ and $X_t$ the corresponding marked conformal structures in
$\Teich(\Sigma)$. Then
$$
\|\bdot X\|_\Teich \le 2\|\bdot g\|.
$$
\end{lemma}

\begin{proof}
  The proof is just a calculation. 
We begin with the case of a $2$--dimensional vector space.

Let $\Conf(V)$ be the space of conformal structures on an oriented
vector space $V$. If $V$ is $2$--dimensional then this can be
identified with orientation preserving $\RR$--linear maps from $V$ to
$\CC$ where two such maps are equivalent if they differ by
post-composition with a $\CC$--linear map. 
Given two conformal structures $\omega_0$, $\omega_1$ in $\Conf(V)$, we define the {\em
  Teichm\"uller distance} between them as follows. Let
$\lambda_0,\lambda_1\colon V\to \CC$ be  $\RR$--linear maps
representing $\omega_0$ and $\omega_1$ and let $\mu =
\frac{(\lambda_1\circ\lambda_0^{-1}))_{\bar
    z}}{(\lambda_1\circ\lambda_0^{-1})_z}$ be the {\em Beltrami
  differential}. Then $d_\Teich(\omega_0,\omega_1) =
\frac12\log\frac{1+|\mu_\lambda|}{1-|\mu_\lambda|}$. 
Note that while $\mu$ depends on the choice of $\lambda_0$ (but not $\lambda_1$), the absolute value $|\mu|$ only depends on $\omega_0$ and $\omega_1$ so $d_\Teich$ is well
defined and one can check that it is a metric.

Let $\omega_t$ be a smooth family in $\Conf(V)$ with a smooth family of representatives $\lambda_t$ and $\mu_t$ the Beltrami differentials between $\lambda_0$ and $\lambda_t$. 
A computation shows that the time zero derivative of the map $t\mapsto d_\Teich(\omega_0,\omega_t)$ is bounded by $|\bdot \mu|$.

An inner product $g$ determines a conformal structure by choosing an
$\RR$--linear map $\lambda\colon V\to \CC$ to be an orientation
preserving isometry from $(V,g)$ to the usual Euclidean metric on
$\CC$. Note that if we multiply $g$ by a scalar we get an equivalent
conformal structure. Now take a smooth family $g_t$ of inner products
and isometries $\lambda_t\colon (V,g_t) \to\CC$. 
If we choose an orthonormal basis for $(V,g_0)$ taken by $\lambda_0$ 
to the standard basis of $\CC$, then the traceless part $[\bdot g]$ of $\bdot g$ is represented by the matrix
\[
[\bdot g] =2\left(\begin{array}{cc} \Re \bdot \mu & \Im \bdot \mu\\ \Im \bdot \mu & - \Re \bdot \mu \end{array} \right).
\]
Another direct computation gives
\begin{equation}\label{mu.dot}
|\bdot\mu| = 2\|[\bdot g]\| \le 2\|\bdot g\|.
\end{equation}

A Riemannian metric $g$ on a surface $\Sigma$ defines a conformal
structure on each tangent space and this defines a conformal structure
on $\Sigma$, and hence a point in $\Teich(\Sigma)$. 
Given a diffeomorphism $f\colon(\Sigma, g_0) \to (\Sigma, g_1)$, the
pointwise identification $df_p\colon T_p\Sigma\to T_{f(p)}\Sigma$
allows us to compare the conformal structures as above, and in
particular to define a Beltrami differential $\mu_f$ whose absolute
value $|\mu_f|$ is well-defined independently of coordinates. We can
then write 
\begin{equation}\label{Teich.defn}
d_\Teich(g_0,g_1) = \underset{f\in{\mathrm {Diff}}_0(\Sigma)}{\inf}
\frac12\log\frac{1+\|\mu_f\|_\infty}{1-\|\mu_f\|_\infty}
\end{equation}
where $\mathrm{Diff}_0(\Sigma)$ is the space of diffeomorphisms of
$\Sigma$ isotopic to the identity. This defines a pseudometric on the
space of Riemannian metrics on $\Sigma$. The quotient metric space is
the Teichm\"uller space $\Teich(\Sigma)$, which can be given a
differentiable structure so that $d_\Teich$ is a Finsler metric,
associated to a norm $\|\cdot\|_\Teich$ on the tangent space.
If $g_t$ is a smooth family of Riemannian metrics on $\Sigma$, and $X_t$
are the corresponding marked conformal structures in $\Teich(\Sigma)$,
then by differentiating (\ref{Teich.defn}) we obtain 
$\| \bdot X\|_\Teich \le \|\bdot\mu\|_\infty$.
Now using (\ref{mu.dot}) we complete the proof of Lemma \ref{bound.Teich.norm}.
\end{proof}

\section{Thick collar}
\label{thick.collar.section}

In this section we let $X_t$ and $M_t$ be as in the
introduction. Let $\core(M_t)$ denote the convex core of $M_t$. 
Our goal is the following statement:

\begin{theorem}[Thick collar in $\core(M_t)$]\label{thick.collar.theorem}
There exists $t_1>0$ and $\delta>0$ depending on $S,\ep$ such that for $t>t_1$ there is
a collar neighborhood of $\boundary\core(M_t)$ in $\core(M_t)$ which
is $\ep_2$--thick and contains a $\delta t$--neighborhood of the
boundary. 
\end{theorem}

\subsubsection*{Uniform contraction}

McMullen showed \cite{McMullen.1990} that if $M$ is a compact hyperbolic $3$--manifold with totally geodesic boundary, then $\sigma_M$ is uniformly contracting.
Namely, if $\dee \sigma_M$ is the derivative of $\sigma_M$ and $\| \,
\dee\sigma_M \, \|_\Teich$ its Teichm\"uller norm as in Section \ref{families.section}, then
\[
\| \, \dee\sigma_M \, \|_\Teich < c_M < 1
\]
over the entire Teichm\"uller space for some constant $c_M$ depending on $M$.
Remarkably, the proof provides uniform contraction independent of $M$.
\begin{theorem}[McMullen \cite{McMullen.1990}]\label{McMullen.contraction.theorem}
There is a constant $c_S$ such that if $M$ is any hyperbolic $3$--manifold with totally geodesic boundary $\partial M \cong S$, then 
\[
\| \, \dee \sigma_M \, \|_\Teich < c_S < 1
\]
at every point in $\Teich(S)$.
\end{theorem}
\begin{proof}[Remark on the proof.]
The proof of uniform contraction in \cite{McMullen.1990} makes very
little use of the topology of the $3$--manifold $M$, and relies only
on the facts that $M$ is compact, irreducible, acylindrical,
atoroidal, and boundary incompressible. The main argument, in the
proof of Theorem 6.1 of \cite{McMullen.1990}, obtains uniform
contraction by considering the possible geometric
limits of a sequence of potential counterexamples. This argument works
just as well if one allows the underlying $3$--manifolds to vary over the
sequence while fixing the topological type of the boundary.
\end{proof}

\subsubsection*{Thick segments}

Let $X_t = \calG(t)$ be as in the hypothesis of the main theorem, and let
$Y_t = \overline{\sigma_M(X_t)}$.  
Note that $Y_0 = \overline{\sigma_M(X_0)} = X_0$.

Theorem \ref{McMullen.contraction.theorem} bounds the speed
\[
\| \, Y'_t \, \|_\Teich < c < 1,
\]
where $c = c_S$ depends only on $S$.  
We conclude that
\begin{equation}\label{slowskins.equation}
\dee(Y_0,Y_t) \leq c \, t
\end{equation}
and so
\[
\dee(X_t,Y_t) \geq (1-c) t
\]
by the triangle inequality.

For $X$, $Y$ in $\Teich(S)$, let $[X,Y]$ denote the Teichm\"uller geodesic
between them.

To produce our thick collar, we begin by showing that the geodesic $[X_t, Y_t]$ has an initial segment $[X_t, Z_t]$ that is $\epsilon_1$--thick for a nice $\epsilon_1$.
To do this, we use the coarse hyperbolicity that Teichm\"uller space exhibits in its thick part.
Theorem \ref{lipschitz.projection.theorem} below, due to Minsky, says that thick geodesics have coarsely Lipschitz closest points projections.  
Theorem \ref{Rafi.thick.thin.theorem} below, due to Rafi, says that geodesic triangles in $\Teich(S)$ try to be thin triangles when they are in the thick part. 
That is, a point in a long thick segment in the side of triangle is close to the union of the other two sides.
Together, these theorems tell us that, since $Y_t$ is far from $X_t$, the geodesic $[X_t,Y_t]$ must fellow travel $[X_t,X_0]$ for a long time.

We now make this precise.

\begin{theorem}[Minsky, Second part of Corollary 4.1 of \cite{Minsky.1996}]\label{lipschitz.projection.theorem}
Let $\epsilon > 0$ and let $S$ be a closed orientable surface.
There is a constant $b$ such that if $\calG$ is an $\epsilon$--thick geodesic in $\Teich(S)$ with closest points projection map $\pi_\calG$, then
\[
\diam(\pi_\calG(X)\cup \pi_\calG(Y)) \leq \dee(X,Y) + b
\]
for any points $X$ and $Y$ in $\Teich(S)$. \qed
\end{theorem}

\begin{theorem}[Rafi, Theorem 8.1 of \cite{Rafi.2014}]\label{Rafi.thick.thin.theorem}
Let $\epsilon > 0$ and let $S$ be a closed hyperbolic surface. 
There are constants $A$ and $B$ such that the following holds.
Let $X$, $Y$, and $Z$ be three points in $\Teich(S)$.
If $[C,D] \subset [X,Y]$ such that $d(C,D) > A$ and every $t$ in $[C,D]$ is $\epsilon$--thick, then there is a $w$ in $[C,D]$ with
\[
\min\{ \dee(w, [X,Z]), \dee(w, [Y,Z]) \} \leq B. 
\quad \quad \quad \quad \quad  \quad \quad \quad \quad \quad \quad \quad \quad \quad \quad \quad \ \,
\qed
\]
\end{theorem}

\begin{theorem}[Minsky, Theorem 4.2 of \cite{Minsky.1996}]\label{Minsky.fellow.travelers.theorem}
Let $K, C, \epsilon > 0$.  
Then there is a $D > 0$ such that the following holds. Let $\calP$ be a $(K,C)$--quasigeodesic path in $\Teich(S)$ whose endpoints are connected by an $\epsilon$--thick Teichm\"uller geodesic $\calG$. 
Then $\calP$ remains in a $D$--neighborhood of $\calG$.
\qed
\end{theorem}

\begin{lemma}[Big thick segment]\label{thick.segment.lemma}
There exist $\ep_1>0$ and $\delta_1 > 0$ depending only on $S$ and $\epsilon$ such that the segment $[X_t, Y_t]$ contains a point $Z_t$ such that
\[
\dee(X_t, Z_t) > \delta_1 t
\]
and $[X_t, Z_t]$ is $\ep_1$--thick. 
\end{lemma}

\begin{figure}
\input{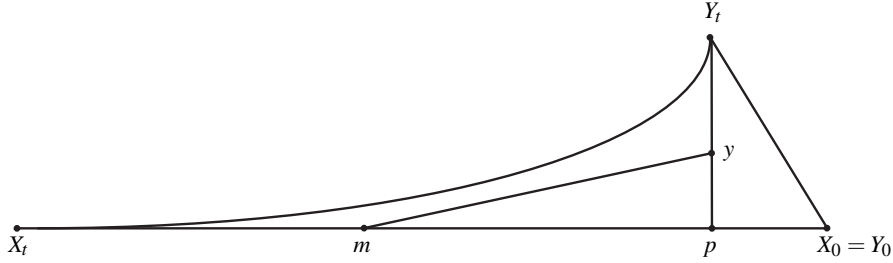}
\caption{Geodesics and points of interest in the proof of Lemma \ref{thick.segment.lemma}.}\label{thick.segment.figure}
\end{figure}
\begin{proof}
See Figure \ref{thick.segment.figure} for geodesics and points of interest throughout the proof.

Let $p$ be a nearest point to $Y_t$ on the geodesic $\calG$, and let $m$ be the midpoint of the geodesic joining $p$ to $X_t$.

Applying Theorem \ref{lipschitz.projection.theorem} to the nearest--points projection $\pi_\calG$, we find that $d(X_0,p) \le d(X_0,Y_t) + b < ct + b$.
Thus $d(p,X_t) > (1-c)t - b$, and so $d(m,p) > ((1-c)t -b)/2$.

Let $y$ in $[p,Y_t]$ minimize distance from $[p,Y_t]$ to $m$ and note that $d(m,y) \ge d(p,y)$. 
The triangle inequality gives us $d(m,p) \le d(m,y) +d(y,p) \le 2d(m,y)$.

Putting these two paragraphs together we obtain $d(m,[p,Y_t]) > ((1-c)t - b)/4$. 
Applying Theorem \ref{Rafi.thick.thin.theorem} to the triangle $\triangle pX_tY_t$ we
can conclude that the point $m$ is a bounded distance $B$ from a point $Z_t$ on the geodesic $[X_t, Y_t]$, where $B$ is a nice constant.

By the triangle equality, this point $Z_t$ is at a distance at least $(1-c)t/2 - B'$ from $X_t$, where $B' = B + b/2$. 
This gives $d(Z_t,X_t) > \delta_1t$ for suitable $\delta_1$ and $t$ larger than a nice constant.
For small $t$ we take $Z_t = Y_t$. 

By Theorem \ref{Minsky.fellow.travelers.theorem}, the geodesic $[Z_t, X_t]$ lies in a $D$--neighborhood of the geodesic $[m, X_t]$ for some nice $D$ (the path $[m, Z_t] \cup [Z_t, X_t]$ is a $(1,2B)$--quasigeodesic).
Since $\calH$ (and hence $[m,X_t]$) is $\epsilon$--thick, there is then a nice $\epsilon_1$  such that $[Z_t, X_t]$ is $\epsilon_1$--thick.
\end{proof}

\subsubsection*{Thick collar}

Let $\core(X_t, Y_t)$ be the convex core of the quasifuchsian manifold $\qf(X_t, Y_t)$ and let $\calX_t$ and $\calY_t$ be the components of $\partial\core(X_t, Y_t)$ facing $X_t$ and $Y_t$, respectively.

We say a subset $\calA$ of a hyperbolic manifold is {\em $\ell$--thick} if it
is contained in the $\ell$-thick part.

\begin{lemma}[Thick collar in $\qf(X_t,Y_t)$]\label{QF.collar.lemma}
There are constants $t_0$, $\epsilon_2$, and $\delta_2$ depending only
on $S$ and $\epsilon$

such that $\core(X_t, Y_t)$ contains an $\epsilon_2$--thick  submanifold $B_t \cong S \times [0,1]$ such that $S \times \{0\} = \calX_t$ and $\dee(\calX_t, S \times \{1\}) \geq \delta_2 t$.
\end{lemma}

\begin{proof}
The key point is that the thick subsegment $[X_t,Z_t]$ from Lemma \ref{thick.segment.lemma} is reflected in
the structure of the model manifold of 
\cite{Brock.Canary.Minsky.ELC2}. To explain this, let $s:\Teich(S)\to\curv(S)$ be the systole map from Teichm\"uller
space to the complex of curves. Then $s$ takes a Teichm\"uller
geodesic to an unparameterized quasigeodesic, as in
\cite{Masur.Minsky.1999}, with the quality of the quasigeodesic
depending only on $S$. Let $\calH$ be a $\curv(S)$--geodesic connecting
$s(X_t)$ to $s(Y_t)$. By hyperbolicity of $\curv(S)$ there is $K=K_S$ 
so that $\calH$ is at Hausdorff distance at most $K$ from
$s([X_t,Y_t])$, and we can find an initial segment $\calH_1$ 
of $\calH$ which lies Hausdorff distance at most $K$ from $s([X_t,Z_t])$.

Because $[X_t,Z_t]$ is $\epsilon_1$--thick, by \cite{Rafi.2005} there is a bound
$d_W(X_t,Z_t) < B$ for $B=B(S,\epsilon_1)$ and any subsurface $W$ of
$S$. (Here, $d_W(X,Y)$ is the ``subsurface projection distance'' discussed
in \cite{Masur.Minsky.2000,Minsky.2010} and
\cite{Brock.Canary.Minsky.ELC2}. Namely it is the distance in the
arc/curve complex of $W$ between the intersections with $W$ of the
shortest filling curve systems in $X$ and $Y$ respectively.)
The Bounded Geodesic Image Theorem \cite{Masur.Minsky.2000}
provides constants $b,k$ such that if the $\curv(S)$--distance between
$[\boundary W]$ and $s([Z_t,Y_t])$ is at least $k$
then $d_W(Z_t,Y_t) \le b$. It follows, after
trimming the end of $\calH_1$ by a bounded amount, that for any $W$ with
$d(\boundary W,\calH_1) \le 1$ we have $d_W(X_t,Y_t) \le c'$.

The Bilipschitz Model Theorem of
\cite{Brock.Canary.Minsky.ELC2,Minsky.2010} provides a manifold $\MM$
depending on $(X_t,Y_t)$ and a bilipschitz homeomorphism $f:\MM\to \qf(X_t,Y_t)$.
The structure of $\MM$ is determined by $\calH$, and in particular any 
vertex $v$ of $\calH$ can be associated to a ``cut surface'' $\tau_v$ in $\MM$, 
whose inclusion is a homotopy equivalence and whose geometry is
determined by $v$ and the projections
$d_W(X_t,Y_t)$ for subsurfaces $W$ with $d(\boundary W,v)\le 1$.
When two vertices are sufficiently far apart their cut surfaces cobound a
product region. The bounds on $d_W(X_t,Y_t) $ from the previous
paragraph imply that product regions determined by cut surfaces based
on vertices of $\calH_1$ have bounded geometry, and in particular are
$\epsilon'$--thick for some nice $\epsilon'$. See Sections 4 and 5, and
particularly Lemma 5.7, of 
\cite{Brock.Canary.Minsky.ELC2} for the construction of these regions
in the general setting. Theorem 7.1 of 
\cite{Brock.Canary.Minsky.ELC2} indicates how bounds on $d_W$ give
rise to bounded--geometry regions.

The initial cut surface of $\calH_1$ is the surface in the model that maps
to the convex hull boundary $\calX_t$. If we build a product
region $B(\calH_1)$ bounded by the initial and final vertices of $\calH_1$,
the distance between its two boundary components is at least  a
uniform multiple of the length $|\calH_1|$ of $\calH_1$.
One can see this by dividing it up using cut surfaces for equally--spaced vertices of $\calH_1$.
Since the length $|\calH_1|$ is a uniform multiple of $t$, the image of $B(\calH_1)$ in $\qf(X_t,Y_t)$ under the bilipschitz model map $f$ is the desired product region.
\end{proof}

It will be convenient to talk about product subregions of $B_t$. From
the fact that $B_t$ has bounded geometry (or the construction itself)
one has for each $x$ in $B_t$ a surface $F_x$ isotopic to $\calX$  which
contains $x$ and has diameter bounded by $D$ depending only on $\ep$.
It follows that for each $s$ in $[0,\delta_2 t]$ there exists a region
$B_t[s]\subset B_t$ such that
\begin{enumerate}
\item $\calN_s(\calX) \subset B_t[s] \subset \calN_{s+2D}(\calX)$,
\item $B_t[s]$ is homeomorphic to $S\times [0,1]$
\end{enumerate}
(here $\calN_s$ denotes an $s$--neighborhood within
$\core(X_t,Y_t)$). Simply pick the region between $F_x$ and $\calX$,
where $\mathrm{dist}(x,\calX) = s+D$.  

Let $\pi:\qf(X_t,Y_t) \to M_t$ be the covering map and 
let $\calX_t' = \pi(\calX_t$).

\begin{lemma}[Embedding of collar]\label{thick.collar.lemma}
There exists $t_1>t_0$ depending on $S,\ep$, such that for $t>t_1$ the
covering map $\pi$ embeds $B_t[\delta_2t/3]$ in $\core(M_t)$, and the
image is in the $\ep_2$--thick part of $M_t$.
\end{lemma}

\begin{proof}
Note first that $\core(X_t,Y_t)$ is contained in the
pullback $\pi^{-1}(\core(M_t))$, and that $\calX_t$ is a boundary
component of both $\core(X_t,Y_t)$ and $\pi^{-1}(\core(M_t))$.
Therefore any component 
$\calZ$ of $\pi^{-1}(\pi(\calX_t))$ cannot meet $int(\core(X_t,Y_t))$, and
if $\calZ \ne \calX_t$ then $\calZ$ is disjoint from $B_t$, which then
separates it from $\calX_t$. It follows 
that the distance from $\calZ$ to $\calX_t$ is at least $\delta_2t$.

Thus for $s<\delta_2t/2$, the $s$--neighborhood $C_s$ of $\calX_t$ in 
$\core(X_t,Y_t)$ is disjoint from the $s$--neighborhoods of the
other components of $\pi^{-1}(\pi(\calX_t))$.
We conclude that $\pi|_{C_s}$ is an embedding into $M_t$. For
suitable $t_1$ we have that the product region $B_t[\delta_2t/3]$ is
in such a neighborhood, so $\pi$ embeds it.

Now as soon as $\delta_2t/6 > \ep_2+2D$ we find that any loop of length $\ep_2$
based at a point in $\pi(B_t[\delta_2t/3])$ lifts to a loop in $B_t[\delta_2t/2]$, so
since $B_t$ is in the $\ep_2$--thick part of $\qf(X_t,Y_t)$, we
conclude that $\pi(B_t[\delta_2t/3])$  is in the $\ep_2$--thick part of
$M_t$. 
\end{proof}

Theorem \ref{thick.collar.theorem} is now just a rewording of Lemma
\ref{thick.collar.lemma}.

\section{Geometric inflexibility}\label{Geometric.inflexibility.section}
The goals of this section are Theorem \ref{eta.C1.control}, which
uses Geometric Inflexibility to give exponentially shrinking bounds on
the time and space derivatives of our family of metrics; and
Proposition \ref{control.skinning}, which uses these bounds and the
proxy surfaces of Lemma \ref{nearby.proxy.lemma} to control the speed
of the skinning image.  

It is a classical fact that two hyperbolic $3$--manifolds are
$K$--quasiconformally conjugate if and only if there is an
$L$--bi-Lipschitz map between them and each of the constants $K$ and
$L$ can be effectively controlled in terms of the other, see, for example, Theorem 2.5 and Corollary B.23 of \cite{McMullen.1996}. 
McMullen \cite{McMullen.1996} showed that if the injectivity radius is bounded away from zero, the bi-Lipschitz map may be chosen so that the pointwise bi-Lipschitz constant decays exponentially to 1 as the point moves deeper into the convex core. 
McMullen called this {\em geometric inflexibility}. 
In Brock--Bromberg \cite{Brock.Bromberg.2011}, an alternative approach to geometric inflexibility removes the global restriction on the injectivity radius and shows that the bi-Lipschitz constant decays exponentially away from the thin part.
 We use this version here.

\begin{theorem}\label{eta.C1.control}
Let $M$ be a compact, smooth hyperbolizable $3$--manifold  with
boundary and $X_t$ a smooth $1$--parameter family of conformal
structures on $\partial M$ such that
$\|\bdot X_t\|_\Teich\le 1$. Then there exists a smooth
family of complete hyperbolic metrics $g_t$ on the interior of $M$
that extend continuously to the conformal structures $X_t$ on
$\partial M$ and such that for $x$ in the $\epsilon$-thick part of $M_t$,
    \[
     \|\eta_t(x)\| \le Ae^{-Bd(x, M_t-\core(M_t))}
     \]
    and
    \[
     \|\nabla^t\eta_t(x)\|\le Ae^{-Bd(x, M_t-\core(M_t))}.
     \]
The constants $A$ and $B$ depend only on the topological type of
$\partial M$ and on $\epsilon$.
\end{theorem}

\begin{proof} The proof is a straightforward combination of several
  results.  Work of Reimann \cite{Reimann.1985} supplies a family of
  hyperbolic metrics $g_t$ that extend continuously to $X_t$ and such
  that the associated strain fields $\eta_t$ are {\em harmonic}. One
  obtains bounds on the $L^2$--norm of $\eta_t$ in $\core(M_t)$ that
  are linear functions of the genus of $\partial M$ (Lemma 5.2 in
  \cite{Brock.Bromberg.2011}). By Theorem 3.6 of
  \cite{Brock.Bromberg.2011}, the $L^2$--norm of $\eta_t$ on the
  submanifold of points in $\core(M_t)$ a distance $>r$ from $\partial
  \core(M_t)$ decays exponentially in $r$. From this one obtains bounds on the
  $L^2$--norm of $\eta_t$ in an $\epsilon$--ball centered at $x$. The
  pointwise norm bounds on $\eta$ and $\nabla\eta$ then follow from standard estimates in partial
  differential equations, see \cite{Brock.Bromberg.2004.erratum}. 
\end{proof}

\subsubsection*{Peripheral curves do not get short}
By assumption the length of any closed curve on $X_t$ is at least $\epsilon$. However, a lower bound on the length of curve on the conformal boundary does not, in general, imply lower bounds on length in the hyperbolic $3$--manifold. We now combine Theorem \ref{thick.collar.theorem} and geometric inflexiblity to show that such a bound does hold for the manifolds in our family.
\begin{theorem}\label{only.nonperipheral.thin}
There exists $\epsilon'>0$ depending on $S$ and $\epsilon$ such that for
all $t>0$ every curve $\gamma$ in $S$ has $\ell_{M_t}(\gamma) > \epsilon'$.
\end{theorem}

\begin{proof} Let $t_1$ be the constant from Lemma \ref{thick.collar.lemma}. Then for all $t< t_1$ there is an $\epsilon_3$, depending only on $\epsilon$ and $t_1$, such that $\ell_{M_t}(\gamma)\ge \epsilon_3$.

Let $\epsilon_2$ also be the constant from Lemma \ref{thick.collar.lemma} and choose $\epsilon_4$ to be the minimum of $\epsilon_2$, $\epsilon_3$ and the $3$--dimensional Margulis constant. If $\ell_{M_t}(\gamma) < \epsilon_4$, let $\bT_t(\gamma)$ be the $\epsilon_4$--Margulis tube for $\gamma$ in $M_t$. Then by Lemma \ref{thick.collar.lemma}, $d(\bT_t(\gamma), M_t - \core(M_t)) \ge \delta_2t/3 = \delta_3 t$ where, again, $\delta_2$ is from Lemma \ref{thick.collar.lemma}.

By Theorem 5.8 in Brock--Bromberg \cite{Brock.Bromberg.2011}, there
exist constants $C_1$ and $C_2$, depending only on $\boundary M$,
such that if $\ell_{M_t} < \epsilon_4$ then
\begin{equation}\label{length.bound}
\left| \log\frac{\ell_{M_{t+s}}(\gamma)}{\ell_{M_t}(\gamma)}\right| \le C_1e^{-C_2d(\bT_t(\gamma), M_t-\core(M_t))}
\end{equation}
for $|s| \le 1$. 
(This is a consequence of their geometric inflexibility theorem,
applied to the boundary of $\bT_t(\gamma)$.)
Choose $\epsilon'<\epsilon_4$ such that

\[
-\log \frac{\epsilon'}{\epsilon_4} = \frac{C_1e^{-C_2\delta_3t_1}}{1-e^{-C_2\delta_3}}.
\]
We will show that $\ell_{M_t}(\gamma) \ge \epsilon'$.

If $\ell_{M_t} (\gamma)\ge\epsilon_4$ we are done.
So assume to the contrary that $\ell_{M_t} (\gamma) < \epsilon_4$.
Choose $t_\gamma< t$ such that $\ell_{M_{t_\gamma}}(\gamma) =
\epsilon_4$ and $\ell_{M_s}(\gamma) \le \epsilon_4$ for all $s$ in $[t_\gamma,t]$. 
Since $\ell_{M_s}(\gamma) \ge \epsilon_4$ when $s\le t_1$ and $\ell_{M_s}(\gamma)$ is continuous in $s$, such a $t_\gamma$ exists and is bigger than $t_1$ . 

Using the fact that for $s$ in $[t_\gamma,t]$ we have $d(\bT_s(\gamma), M_s-\core(M_s)) \ge \delta_3s$, we can repeatedly apply \eqref{length.bound} to see that
\[
\left| \log \frac{\ell_{M_t}(\gamma)}{\ell_{M_{t_\gamma}}(\gamma)}\right| \le \sum_{k=0}^n C_1e^{-C_2 \delta_3(t_\gamma+k)}<\sum_{k=0}^\infty C_1e^{-C_2\delta_3 (t_\gamma+k)}
\]
where $n$ is the least integer greater than $t-t_\gamma$. Summing this geometric series gives the desired bound.
\end{proof}

\section{Proxy surfaces}\label{proxy.section}

We now introduce the smooth locally convex surfaces in $M_t$ whose
geometry will give us good control of the conformal geometry of the
skinning surface $Y_t$.
In Lemma \ref{first.time.lemma}, we locate these surfaces with respect to
our thick collars. 

\begin{lemma}\label{nearby.proxy.lemma} There is a smooth surface $\calE_t$ in the $Y_t$--end of $\qf(X_t, Y_t)$ in the $3$--neighbor-hood of $\calY_t$ whose principal curvatures are within $\frac{1}{4}$ of $1$.
\end{lemma}
\begin{proof}
Observe that there is a smooth convex surface arbitrarily close to $\calY_t$. One can construct such a surface in several ways. For example, one can smoothly approximate the distance function from $\calY_t$ by convex functions and take a level set. A more concrete construction is due to Labourie (\cite{Labourie.1991}) who showed that, for any $\kappa$ in $(0,-1)$, there is a surface ${\mathcal L}_\kappa$ of constant Gaussian curvature $\kappa$ in the $Y_t$--end of $\qf(X_t, Y_t)$. As $\kappa\to -1$ the surfaces $\mathcal L_\kappa$ will converge uniformly to $\calY_t$. If we flow any convex surface a distance $r$ in the normal direction then the curvatures are bounded between $\tanh r$ and $\coth r$. We then obtain $\calE_t$ by flowing the smooth convex surface near $\calY_t$ a distance $2$.\end{proof}

\begin{lemma}\label{first.time.lemma}
There is a time $t_2 > t_1$ depending only on $S$ and $\epsilon$ such
that, for all $t \geq t_2$, the surface $\calE_t'=\pi(\calE_t)$ lies in
$\core(M_t)\ssm \pi(B_t[\delta_2t/4])$.
\end{lemma}
\begin{proof}
The surface $\calY_t$ is $\ep'$--thick by Theorem
\ref{only.nonperipheral.thin} and so the diameter of $\calY_t$ is bounded by a nice constant.

Since the covering map $\pi$ maps $\core(X_t,Y_t)$ into $\core(M_t)$,
the image $\calY_t'$ of $\calY_t$ lies in $\core(M_t)$.
If $\calY_t'$ lies entirely in $\pi(B_t[\delta_2t/3])$
then $\calY_t$
lies in a component of $\pi^{-1}(\pi(B_t[\delta_2t/3]))$, and as in Lemma
\ref{thick.collar.lemma} these components are retracts of the lifts of
$\calX'_t$. All of them except $B_t[\delta_2t/3]$ itself are simply
connected (since $M$ is acylindrical) and
thus cannot contain $\calY_t$.
Furthermore, $B_t[\delta_2t/3]$ cannot contain $\calY_t$
since it is inside $\core(X_t,Y_t)$. 

We conclude that $\calY_t'$ cannot lie in $B_t[\delta_2t/3]$. 
So, if $t$ is sufficiently large (depending on the diameter bound for $\calY_t$), then $\calY_t'$ will be disjoint from $\pi(B_t[\delta_2t/4])$.

Since $\calE'_t$ is in a $3$--neighborhood of $\calY_t'$, it is also disjoint
from $\pi(B_t[\delta_2t/4])$ when $t$ is large enough.
\end{proof}

\subsubsection*{Horocylically convex surfaces and their conformal structures}
Let $\Sigma$ be a transversally oriented surface
immersed in a hyperbolic $3$--manifold $(M,g)$. This gives a normal vector field ${\mathbf n}$ to
$\Sigma$, and a shape operator $B:T\Sigma \to T\Sigma$ given by
$B(x) = \nabla_{x} {\mathbf n}$. If the eigenvalues of $B$ lie in
$(-1,\infty)$ (i.e. the principal curvatures are bigger than $-1$), then $\Sigma$ is {\em horocyclically convex} and the geodesic flow to infinity along $\nn$ gives a complex
structure $\omega$ on $\Sigma$ (in fact a complex projective
structure). More precisely let $\Sigma_r$ be the surface obtained by flowing $\Sigma$ in the direction of $\nn$ a distance $r$.  The condition that $\Sigma$ is horocylically convex is equivalent to this normal flow being non-singular for all $r\ge 0$. 
If we pull the metrics on $\Sigma_r$ back to $\Sigma$ by the normal flow, the metrics diverge, but the conformal structures converge to a conformal structure $Y_\Sigma$.

 In our setting, $\Sigma$ is the locally convex surface $\calE'_t$ and the conformal structure is the skinning surface $Y_t$.

A $1$--parameter family of hyperbolic metrics on $M$ determines a $1$--parameter family of conformal structures on $\Sigma$. We want to convert bounds on the derivative of the metric to bounds on the derivative of the conformal structures in Teichm\"uller space. The key to this is the following formula which gives the conformal structure $Y_\Sigma$ in terms of the geometry of $\Sigma$.

\begin{lemma}[Krasnov--Schlenker \cite{KS08}]\label{conformal.structure.infty}
Let $\Sigma$ be a horocyclically convex surface in a hyperbolic $3$--manifold $(M,g)$ with first fundamental form $I = g|_\Sigma$ and shape operator $B$. Then $I^*(x,y)  = I(x+Bx,y+By)$ is a Riemannian metric on $\Sigma$ in the conformal class $Y_\Sigma$.
\end{lemma}

If $g_t$ is a smooth family of complete hyperbolic metrics on $M$ we obtain a
family of shape operators $B_t$ and conformal structures $\omega_t$. At
each $t$ we have a strain field $\eta_t$ defined as before. We wish to
control the speed of $\omega_t$ in $\Teich(\Sigma)$ in terms of the
behavior of $B_t$ and $\eta_t$.

\begin{proposition}\label{control.skinning}
Let $(M, g_t)$ be a manifold with a smooth family $g_t$ of complete
hyperbolic metrics. Let $\Sigma$ be a closed immersed transversally
oriented surface in $M$ and let $\omega_t$, $B_t$ and $\eta_t$ be the
conformal structure, shape
operator and strain field for $g_t$, respectively. Given $k$ there exists $C$ such
that, if the eigenvalues of $B_0$ lie in $[-1+1/k,k]$, then
 \[
  \|\bdot\omega\|_\Teich < C\, \max\left(
  \|\eta\|_{g},
\|\grad \eta\|_{g}\right).
\]
\end{proposition}

\begin{proof}
Let $x,y$ and $z$ be tangent vector fields on $\Sigma$.
Differentiating the formula from Lemma \ref{conformal.structure.infty} we have
\begin{equation}\label{I.star.dot}
\bdot{I^*}(x,y) = 2I(\eta(x+Bx), y+By) + I(\bdot Bx ,y+By) +
I(x+Bx,\bdot By).
\end{equation}
From Lemma \ref{bound.Teich.norm} we see that a bound on $\|\bdot{I^*}\|$
for all points in $\Sigma$ gives a bound on $\|\bdot \omega\|_\Teich$.  
From (\ref{I.star.dot}) we have that given a bound on $\|B\|$, $\|\bdot{I^*}\|$ is bounded by a linear function of $\|\eta\|$ and $\|\bdot B\|$.
If $\nabla^t$ is the Riemannian connection for $g_t$ and $\nn_t$ is the
unit normal outward vector field for $(\Sigma, g_t)$, then $B_t x =
\nabla^t_x \nn_t$.
Therefore $\bdot B x = \nabla_x \bdot \nn + \bdot\nabla_x \nn$, and so we need to control $\nabla \bdot \nn$ and $\bdot\nabla$.

Given a vector $\vv$ at a point in $\Sigma$, we let $\vv\tangent$ be the component of $\vv$ tangent to $\Sigma$.

First consider $\nabla\bdot \nn$. 
We only need to bound $\nabla\bdot \nn \tangent$  as we are taking the inner
product against tangent vectors. 

We begin by differentiating the
formula $g_t(\nn_t,y) = 0$ to see that 
\[
2g(\eta \nn,y) +g(\bdot \nn,y)=0.
\]
Note that this implies that $g(2\eta \nn + \bdot \nn, y) = 0$ and so $2\eta \nn+\bdot \nn$ is orthogonal to $\Sigma$. We will use this later. 
Differentiating in the $x$--direction we have
\begin{eqnarray*}
0 &=&
x(2g(\eta \nn,y) +g(\bdot \nn,y)) \\& = & 2g(\nabla_x(\eta \nn),y) + 2g(\eta \nn,\nabla_xy) + g(\nabla_x\bdot \nn,y)+ g(\bdot \nn,\nabla_xy)\\
&=& 2g((\nabla_x\eta)\nn, y) +2g(\eta(Bx),y) + g(2\eta \nn+\bdot \nn, \nabla_xy)+g(\nabla_x\bdot \nn,y).
\end{eqnarray*}
As $2\eta \nn + \bdot \nn$ is normal we only need to know the normal component of $\nabla_xy$. 
Since $g(\nabla_xy,\nn) + g(y, \nabla_x\nn) = xg(y,\nn) = 0$, we have $g(\nabla_xy, \nn) = -g(y, Bx)$ and so
\[
|g(2\eta \nn + \bdot \nn, \nabla_xy)| = \|2\eta \nn + \bdot \nn\||g(y,Bx)|.
\]
Combining we have
\[
\|\nabla_x\bdot\nn \tangent \| \le 2\|\nabla\eta\| + 4\|\eta\|\|B\| + \|\bdot \nn\|\|B\|.
\]

We now bound $\bdot \nn$. 
Let $x$ be a unit vector in the direction $\bdot \nn\tangent$. 
Differentiating the formula $g_t(x, \nn_t) = 0$, we have
\[
2g(\eta x, \nn) + g(x,\bdot \nn) = 0
\]
and so $|g(x, \bdot \nn)| \le 2\|\eta\|$.
Differentiating $g_t(\nn_t, \nn_t) = 1$, we see that
\[
2g(\eta \nn, \nn) + 2g(\bdot \nn, \nn) = 0
\]
and so $|g(\bdot \nn,\nn)| \le \|\eta\|$. Therefore $\|\bdot \nn\| \le 3\|\eta\|$.

To bound $\bdot\nabla$ we differentiate the formula $xg_t(y,z) = g_t(\nabla^t_x y, z) + g_t(y, \nabla^t_x z)$. The left hand side is
\[
2xg(\eta y, z) = 2\left(g(\nabla_x (\eta y), z\right)+ g\left(\eta y, \nabla_x z)\right)
\]
and the right hand side is 
\[
2g\left(\eta(\nabla_xy),z\right) + g(\bdot\nabla_x y,z) +2g(\eta y, \nabla_xz) + g(y,\bdot\nabla_xz).
\]
Rearranging and applying the Leibnitz rule to $\nabla_x(\eta y)$, this becomes
\begin{equation}\label{derivativeofstrain}
2g((\nabla_x\eta)y,z) =g(\bdot\nabla_xy,z) + g(y, \bdot\nabla_xz).
\end{equation}
As the Riemannian connections are torsion free we have 
\[
\nabla^t_xy - \nabla^t_yx = [x,y]
\]
and differentiating we see that $\bdot\nabla_x y = \bdot\nabla_yx$. 
Taking the three permutations of \eqref{derivativeofstrain}, the symmetry of $\bdot\nabla$ gives
\[
g(\bdot\nabla_xy,z) = g((\nabla_x\eta)y,z) + g((\nabla_y\eta)z,x) - g((\nabla_z\eta) x,y),
\]
and so $\|\bdot\nabla\| \le 3\|\nabla\eta\|$.

Combing the bounds on $\|\nabla \bdot \nn\|$ and $\|\bdot\nabla\|$ we have
\begin{eqnarray*}
\|\bdot B\| &\le& 2\|\nabla\eta\|+4\|\eta\|\|B\| + 3\|\eta\|\|B\| + 3\|\nabla\eta\|\\
& \le & 5\|\nabla\eta\| + 7\|\eta\|\|B\|.
\end{eqnarray*}
\end{proof}

\section{Finishing the proof}
Let $M$ be a hyperbolizable acylindrical $3$--manifold and assume that
$X$ is the conformal boundary of the unique hyperbolic structure on
$M$ whose convex core boundary is totally geodesic. Let $X_t$ be an
$\epsilon$--thick Teichm\"uller geodesic ray in $\Teich(\partial M)$
with $X = X_0$. Let $M_t = (M^\circ, g_t)$ be the hyperbolic metrics given
by Theorem \ref{eta.C1.control} with conformal boundary $X_t$ and let $Y_t
=\overline{\sigma_M(X_t)}$ be the skinning surface with its orientation reversed.
 By Lemma \ref{nearby.proxy.lemma}, there are convex surfaces $\calE_t$ in $\qf(X_t,Y_t)$ 
 with curvatures within $\frac14$ of $1$ and whose conformal structures at
infinity are $Y_t$.
By Lemma \ref{first.time.lemma}, the image $\calE'_t$ of $\calE_t$ in
$M_t$ is contained in $\core(M_t)$ and $d(\calE'_t,
M_t-\core(M_t)) \ge \delta_2 t$. 
By Theorem \ref{eta.C1.control} 
we have
 \[
 \|\eta_t(x)\| \le Ae^{-B\delta_2 t} \quad \mathrm{and} \quad \|\nabla^t \eta_t (x)\| \le Ae^{-B\delta_2 t}
 \]
 for any $x$ in $\calE'_t$.
 By Proposition \ref{control.skinning} we have
 \begin{equation}\label{final.estimate}
 \|\dot Y_t\|_\Teich < ACe^{-B\delta_2 t}.
 \end{equation}
 All of these constants are nice, and Theorem \ref{BoundAlongThickRays.theorem} follows by integrating \eqref{final.estimate}.

\bibliographystyle{plain}
\bibliography{Skin_thick_rays_final.bib}

\begin{thebibliography}{10}

\bibitem{Brock.Bromberg.2011}
Jeffrey Brock and Kenneth Bromberg.
\newblock Geometric inflexibility and 3-manifolds that fiber over the circle.
\newblock {\em J. Topol.}, 4(1):1--38, 2011.

\bibitem{Brock.Bromberg.2004.erratum}
Jeffrey~F. Brock and Kenneth~W. Bromberg.
\newblock Erratum to "{O}n the density of geometrically finite {K}leinian
  groups".
\newblock To appear, \textit{Acta Mathematica}.

\bibitem{Brock.Canary.Minsky.ELC2}
Jeffrey~F. Brock, Richard~D. Canary, and Yair~N. Minsky.
\newblock The classification of {K}leinian surface groups, {II}: {T}he ending
  lamination conjecture.
\newblock {\em Ann. of Math. (2)}, 176(1):1--149, 2012.

\bibitem{Kent.2010}
Autumn {Kent}.
\newblock Skinning maps.
\newblock {\em Duke Math. J.}, 151(2):279--336, 2010.

\bibitem{KS08}
K.~Krasnov and J-M. Schlenker.
\newblock On the renomalized volume of hyperbolic \hbox{3-manifolds}.
\newblock {\em Comm. Math. Phys.}, 279:637--668, 2008.

\bibitem{Labourie.1991}
Fran\c{c}ois Labourie.
\newblock Probl\`eme de {M}inkowski et surfaces \`a courbure constante dans les
  vari\'et\'es hyperboliques.
\newblock {\em Bull. Soc. Math. France}, 119(3):307--325, 1991.

\bibitem{Masur.Minsky.1999}
Howard~A. Masur and Yair~N. Minsky.
\newblock Geometry of the complex of curves. {I}. {H}yperbolicity.
\newblock {\em Invent. Math.}, 138(1):103--149, 1999.

\bibitem{Masur.Minsky.2000}
Howard~A Masur and Yair~N Minsky.
\newblock {Geometry of the complex of curves. II. Hierarchical structure}.
\newblock {\em Geom. Funct. Anal.}, 10(4):902--974, 2000.

\bibitem{McMullen.1990}
Curtis~T. McMullen.
\newblock Iteration on {T}eichm\"uller space.
\newblock {\em Invent. Math.}, 99(2):425--454, 1990.

\bibitem{McMullen.1996}
Curtis~T. McMullen.
\newblock {\em Renormalization and 3-manifolds which fiber over the circle},
  volume 142 of {\em Annals of Mathematics Studies}.
\newblock Princeton University Press, Princeton, NJ, 1996.

\bibitem{Minsky.2010}
Yair Minsky.
\newblock The classification of {K}leinian surface groups. {I}. {M}odels and
  bounds.
\newblock {\em Ann. of Math. (2)}, 171(1):1--107, 2010.

\bibitem{Minsky.1996}
Yair~N. Minsky.
\newblock Quasi-projections in {T}eichm\"uller space.
\newblock {\em J. Reine Angew. Math.}, 473:121--136, 1996.

\bibitem{Rafi.2005}
Kasra Rafi.
\newblock A characterization of short curves of a {T}eichm\"uller geodesic.
\newblock {\em Geom. Topol.}, 9:179--202, 2005.

\bibitem{Rafi.2014}
Kasra Rafi.
\newblock Hyperbolicity in {T}eichm\"uller space.
\newblock {\em Geom. Topol.}, 18(5):3025--3053, 2014.

\bibitem{Reimann.1985}
H.~M. Reimann.
\newblock Invariant extension of quasiconformal deformations.
\newblock {\em Ann. Acad. Sci. Fenn. Ser. A I Math.}, 10:477--492, 1985.

\bibitem{Thurston.1979.Bangor}
William~P. Thurston.
\newblock Hyperbolic geometry and {$3$}-manifolds.
\newblock In {\em Low-dimensional topology (Bangor, 1979)}, volume~48 of {\em
  London Math. Soc. Lecture Note Ser.}, pages 9--25. Cambridge Univ. Press,
  Cambridge, 1982.

\end{thebibliography}

\bigskip

{
\footnotesize

\noindent 
\textsc{University of Utah}

\medskip

\noindent 
\textsc{University of Wisconsin -- Madison}

\medskip

\noindent 
\textsc{Yale University}

}

\end{document}